\documentclass[10pt]{amsart}
 
\usepackage{latexsym}
\usepackage{amsxtra}
\usepackage{amssymb}
\usepackage{bm}
\usepackage{amsthm}
\usepackage{mathtools}
\usepackage{xspace}
\usepackage{graphicx}
\usepackage[margin=1in]{geometry}
\usepackage{xfrac}
\usepackage{euscript}
\usepackage{calrsfs}

\input xy
\xyoption{all} \CompileMatrices

\begin{document}

\def\HH{{\mathcal{H}}}
\def\orb{{\operatorname{orb}}}
\def\diam{{\operatorname{diam}}}
\def\II{{\mathfrak{I}}}
\def\PO{{\operatorname{PO}}}
\def\Cl{{\operatorname{Cl}}}
\def\Max{{\operatorname{-Max}}}
\def\XX{{\bf{X}}}
\def\YY{{\bf{Y}}}
\def\BBB{{\mathcal B}}
\def\inv{{\operatorname{inv}}}
\def\emph{\it}
\def\Int{{\operatorname{Int}}}
\def\Spec{\operatorname{Spec}}
\def\Bin{{\operatorname{B}}}
\def\n{\operatorname{b}}
\def\N{{\operatorname{GB}}}
\def\BC{{\operatorname{BC}}}
\def\dlog{\frac{d \log}{dT}}
\def\Sym{\operatorname{Sym}}
\def\Nr{\operatorname{Nr}}
\def\lbrack{{\{}}
\def\rbrack{{\}}}
\def\burnside{\operatorname{B}}
\def\Sym{\operatorname{Sym}}
\def\Hom{\operatorname{Hom}}
\def\Inj{\operatorname{Inj}}
\def\Aut{{\operatorname{Aut}}}
\def\Mor{{\operatorname{Mor}}}
\def\Map{{\operatorname{Map}}}
\def\CMap{{\operatorname{CMap}}}
\def\GMaps{G{\operatorname{-Maps}}}
\def\Fix{{\operatorname{Fix}}}
\def\res{{\operatorname{res}}}
\def\ind{{\operatorname{ind}}}
\def\inc{{\operatorname{inc}}}
\def\coind{{\operatorname{cnd}}}
\def\Equiv{{\mathcal{E}}}
\def\W{\operatorname{W}}
\def\F{\operatorname{F}}
\def\witt{\operatorname{gh}}
\def\ngh{\operatorname{ngh}}
\def\Fm{{\operatorname{Fm}}}
\def\bij{{\iota}}
\def\mk{{\operatorname{mk}}}
\def\km{{\operatorname{mk}}}
\def\VV{{\bf{V}}}
\def\ff{{\bf{f}}}
\def\ZZ{{\mathbb Z}}
\def\Zhat{{\widehat{\mathbb Z}}}
\def\CC{{\mathbb C}}
\def\PP{{\mathbb P}}
\def\EE{{\mathbb E}}
\def\MM{{\mathbb M}}
\def\JJ{{\mathbb J}}
\def\NN{{\mathbb N}}
\def\RR{{\mathbb R}}
\def\QQ{{\mathbb Q}}
\def\FF{{\mathbb F}}
\def\mm{{\mathfrak m}}
\def\nn{{\mathfrak n}}
\def\jj{{\mathfrak j}}
\def\aaa{{{{\mathfrak a}}}}
\def\bbb{{{{\mathfrak b}}}}
\def\ppp{{{{\mathfrak p}}}}
\def\qqq{{{{\mathfrak q}}}}
\def\PPP{{{{\mathfrak P}}}}
\def\BB{{\mathfrak B}}
\def\jj{{\mathfrak J}}
\def\LL{{\mathfrak L}}
\def\qq{{\mathfrak Q}}
\def\rr{{\mathfrak R}}
\def\DD{{\mathfrak D}}
\def\cc{{\mathfrak S}}
\def\TT{{\mathcal T}}
\def\SS{{\mathcal S}}
\def\UU{{\mathcal U}}
\def\AA{{\mathcal A}}
\def\BB{{\mathcal B}}
\def\Primes{{\mathcal P}}
\def\genS{{\langle S \rangle}}
\def\genT{{\langle T \rangle}}
\def\bT{\mathsf{T}}
\def\bD{\mathsf{D}}
\def\bC{\mathsf{C}}
\def\VV{{\bf V}}
\def\ff{{\bf f}}
\def\uu{{\bf u}}
\def\aa{{\bf{a}}}
\def\bb{{\bf{b}}}
\def\zero{{\bf 0}}
\def\rad{\operatorname{rad}}
\def\End{\operatorname{End}}
\def\id{\operatorname{id}}
\def\mod{\operatorname{mod}}
\def\im{\operatorname{im}\,}
\def\ker{\operatorname{ker}}
\def\coker{\operatorname{coker}}
\def\ord{\operatorname{ord}}

\newtheorem{theorem}{Theorem}[section]
\newtheorem{proposition}[theorem]{Proposition}
\newtheorem{corollary}[theorem]{Corollary}
\newtheorem{conjecture}[theorem]{Conjecture}
\newtheorem{remark}[theorem]{Remark}
\newtheorem{lemma}[theorem]{Lemma}
\newtheorem{example}[theorem]{Example}
\newtheorem{problem}[theorem]{Problem}

 \newenvironment{map}[1]
   {$$#1:\begin{array}{rcl}}
   {\end{array}$$
   \\[-0.5\baselineskip]
 }

 \newenvironment{map*}
   {\[\begin{array}{rcl}}
   {\end{array}\]
   \\[-0.5\baselineskip]
 }

 \newenvironment{nmap*}
   {\begin{eqnarray}\begin{array}{rcl}}
   {\end{array}\end{eqnarray}
   \\[-0.5\baselineskip]
 }

 \newenvironment{nmap}[1]
   {\begin{eqnarray}#1:\begin{array}{rcl}}
   {\end{array}\end{eqnarray}
   \\[-0.5\baselineskip]
 }

\newcommand{\eq}{eq.\@\xspace}
\newcommand{\eqs}{eqs.\@\xspace}
\newcommand{\diagram}{diag.\@\xspace}


\title[Group actions and musical scales]{Group actions, power mean orbit size, and musical scales} 

\author[J.\ Elliott]{Jesse Elliott}
\address{Department of Mathematics\\ California
State University, Channel Islands\\ Camarillo, California 93012}
\email{jesse.elliott@csuci.edu}

\maketitle

\begin{abstract}

We provide an application of the theory of group actions to the study of musical scales.   For any group $G$, finite $G$-set $S$, and real number $t$, we define the {\it $t$-power diameter} $\diam_t(G,S)$ to be the size of any maximal orbit of $S$ divided by the $t$-power mean orbit size of the elements of $S$.  The symmetric group $S_{11}$ acts on the set of all tonic scales, where a {\it tonic scale} is a subset of $\ZZ_{12}$ containing $0$.  We show that, for all $t \in [-1,1]$, among all the subgroups $G$ of $S_{11}$, the $t$-power diameter of the $G$-set of all heptatonic scales is largest for the subgroup $\Gamma$, and its conjugate subgroups, generated by $\{(1 \ 2),(3 \ 4),(5 \ 6),(8 \ 9),(10 \ 11)\}$.    The unique maximal $\Gamma$-orbit consists of the 32  th\=ats of Hindustani classical music popularized by Bhatkhande.   This analysis provides a reason why these 32 scales, among all 462 heptatonic scales, are of mathematical interest.  We also apply our analysis, to a lesser degree, to hexatonic and pentatonic scales.

\ \\



\end{abstract}

\section{Introduction and summary}

This paper provides an application of group actions to the study of musical scales and uses it to motivate a new invariant, which we call the {\it $t$-power diameter} $\diam_t(G,S)$, defined, for any group $G$, finite $G$-set $S$, and extended real number $t \in \RR\cup\{\infty,-\infty\}$, to be the size of any maximal orbit of $S$ divided by the $t$-power mean orbit size of the elements of $S$.

We may represent the {\bf chromatic scale} as the set $$\ZZ_{12} = \{0,1,2,3,4,5,6,7,8,9,10,11\},$$ 
where $0$ represents the {\bf tonic}, or {\bf key}, of the chromatic scale, which can be any fixed pitch class.  Thus, for example, if  one decides to let $0$ represent the pitch class C, then $1$ represents C$\sharp$, $2$ represents D, and so on.  A {\bf scale (in $\ZZ_{12}$)} is a subset of $\ZZ_{12}$, while a {\bf tonic scale (in $\ZZ_{12}$)} is a scale in $\ZZ_{12}$ containing $0$.   A tonic scale is {\bf $k$-tonic} if it consists of $k$ notes, where $k \in \{1,2,\ldots,12\}$; thus, any tonic scale $s$ is $|s|$-tonic, where $|s|$ is the cardinality of $s$.   The $k$-tonic scales, respectively, for $k = 1,2,\ldots,12$ are called {\bf monotonic}, {\bf ditonic},  {\bf tritonic}, {\bf tetratonic},  {\bf pentatonic},  {\bf hexatonic}, {\bf heptatonic},  {\bf octatonic},  {\bf nonatonic}, {\bf decatonic},  {\bf hendecatonic}, and {\bf chromatic}.  There are a total of $2^{11} = 2048$ possible tonic scales, with a total of ${11 \choose k-1}$ $k$-tonic scales for each $k = 1, 2, 3, \ldots, 12$.  These numbers comprise the eleventh row of Pascal's triangle:
$$1 \ \ \ 11 \ \ \  55  \ \ \ 165 \ \ \ 330  \ \ \ 462 \ \ \ 462 \ \ \ 330 \ \ \ 165 \ \ \ 55 \ \ \ 11 \ \ \ 1.$$
Thus, for example, there are 462 heptatonic scales and 330 pentatonic scales.    We call a scale that may not contain the tonic an {\bf atonic scale (in $\ZZ_{12}$)}.    There are a total of $2^{12} = 4096$ possible atonic scales (including the unique empty scale), with a total of ${12 \choose k}$ {\bf $k$-atonic} scales for each $k = 0, 1, 2, 3, \ldots, 12$.  These numbers comprise the twelfth row of Pascal's triangle:
$$1 \ \ \ 12 \ \ \  66  \ \ \ 220 \ \ \ 495  \ \ \ 792 \ \ \ 924 \ \ \ 792 \ \ \ 495 \ \ \ 220 \ \ \ 66 \ \ \ 12 \ \ \ 1.$$

Throughout this paper,  $\TT$ denotes the set of all tonic scales in $\ZZ_{12}$ and $\TT_k$ the set of all $k$-tonic scales in $\ZZ_{12}$.   The symmetric group $S_{11}$ acts naturally on the sets $\TT$ and $\TT_k$: a permutation $\sigma$ in $S_{11}$ acts on a tonic scale $s \in \TT$ by $\sigma \cdot s = \sigma(s) = \{\sigma(x): x \in s\}$, where one sets $\sigma(0) = 0$.    Explicitly, an element $\sigma$ of $S_{11}$ maps a scale $\{0,s_1, s_2, \ldots, s_{k-1}\}$ to the scale $\{0,\sigma(s_1),\sigma(s_2), \ldots, \sigma(s_{k-1})\}$.  This defines an action of $S_{11}$ on $\TT$, and since $|\sigma(s)| = |s|$ for all $s$, the action induces an action on $\TT_k$ for all  $k \in \{1,2,\ldots,12\}$.  More generally, the group $S_{12}$ acts on the set $\SS$ of all (atonic) scales in $\ZZ_{12}$.   Moreover, if we consider $S_{11} = \{\sigma \in S_{12}: \sigma(0) = 0\}$ as a subgroup of $S_{12}$, then the action of $S_{12}$ on $\SS$ descends to the action of $S_{11}$ on $\TT$.  Although in this paper we focus mainly on the action of $S_{11}$ on $\TT$, many of our results ascend appropriately to the action of $S_{12}$ on $\SS$.

The group $S_{11}$ has $11! = 39,916,800$ elements (and the group $S_{12}$ has $12! = 479,001,600$ elements).   A ``musical'' scale acted on by a randomly chosen element of $S_{11}$ is very unlikely to be very musical.  For example, under the permutation $(1\ 4)(3\ 5)(8\ 7 \ 9 \ 10)$, the heptatonic {\bf major scale} $\{0,2,4,5,7,9,11\}$ maps to the somewhat ``unmusical'' scale $\{0,1,2,3,9, 10,11\}$. 
However, by contrast, some permutations preserve musicality fairly well, e.g., the permutation $(3 \ 4)(8 \ 9)$, which swaps the major scale and the {\bf harmonic minor scale} $\{0,2,3,5,7,8,11\}$.   One of the main questions we investigate in this paper is the following:  are there medium-sized subgroups of $S_{11}$ whose actions on ${\TT}_7$ preserve ``musicality''?   An ideal subgroup of $S_{11}$ would be one that ``respects musicality'' in the sense that scales of approximately the same ``musicality'' appear in the same orbit.   As we will see, some subgroups of $S_{11}$ induce actions on the heptatonic scales that preserve musicality better than others do.  One of our main claims is that the group $\Gamma$ of $S_{11}$ generated by $\{(1 \ 2),(3 \ 4),(5 \ 6),(8 \ 9),(10 \ 11)\}$ is the ``best'' such subgroup, and the 32 scales in its unique maximal orbit, which  coincide with the 32 th\=ats of Hindustani (North Indian) classical music, represent under a particular measure the ``most musical'' scales among the 462 possible heptatonic scales.

To measure the musical efficacy of a subgroup $G$ of $S_{11}$, we define the {\bf $G$-musicality} of a scale $s \in \TT_k$ to be the size $|Gs|$ of the $G$-orbit $Gs$ of $s$ in $\TT_k$ divided by the average size of a $G$-orbit of $\TT_k$.  A consequence of this definition is that the $G$-musicality is a small as possible, namely $1$, for all scales if and only if every orbit has the same number of elements, which holds if $G = S_{11}$ or if $G$ is the trivial group.  In general, the $G$-musicality of a given scale will attain a maximum for some subgroups of $S_{11}$ in between those two extremes.   

The intution behind the concept of $G$-musicality is that, if a scale has a small $G$-orbit relative to the average $G$-orbit size, then it has too much symmetry relative to $G$ and thus the notes comprising the scale are more ``$G$-equivalent'' to each other and therefore have fewer notes that have their own individual character, whereas scales with larger orbits relative to the average orbit size are comprised of notes that can be better differentiated, or distiguished from one another, by the group $G$.  The thesis of this paper is that the subgroups $G$ of $S_{11}$ that yield the largest possible $G$-musicality (relative to the other subgroups of $S_{11}$) of any $k$-tonic scale naturally lead to  mathematically and musically interesting theories of $k$-tonic scales, namely, those $k$-tonic scales with the largest $G$-musicality for any subgroup $G$ of $S_{11}$.

There is a slight subtlety here, however, since, given a group $G$ and a finite $G$-set $S$, the average size of a $G$-orbit of $S$ can be measured in at least two distinct ways.  One might naively define the average size of a $G$-orbit of $S$ to be $$\frac{|S|}{|S/G|} = \frac{\sum_{i = 1}^r |O_i|}{r},$$ where $S/G = \{O_1, O_2, \ldots, O_r\}$  is the set of all $G$-orbits of $S$ and where $r = |S/G|$ is the number of orbits.  This represents the average number of elements in each orbit, in a naive sense.  However, one may also define the {\bf average orbit size of the elements of $S$} to be
$$\orb_1(G,S) = \frac{\sum_{s \in S} |Gs|}{|S|} = \frac{\sum_{i = 1}^r |O_i|^2}{\sum_{i = 1}^r |O_i|}.$$  This represents the expected value of $|Gs|$ for $s \in S$, where each element of $S$ is equally likely to be chosen.   By contrast, the number $\frac{|S|}{|S/G|}$ previously considered represents the expected value of $|O_i|$, where each {\it orbit} $O_i$ is equally likely to be chosen.   Since our focus is on the elements of $S$ rather than on the orbits, $\orb_1(G,S)$ is a better notion of average orbit size than is the naive definition $\frac{|S|}{|S/G|}$.

Nevertheless, both of these measures of ``average orbit size'' have mathematical merit, and this is supported by the observation that, using {\it power means}, one may continuously deform one of these two means to the other, as follows.  For any $t \in \RR-\{0\}$, we define the {\bf $t$-power mean orbit size of the elements of $S$} to be
$$\orb_t(G,S) =  \left(\frac{\sum_{s \in S} |Gs|^t}{|S|}\right)^{1/t}.$$    This represents the $t$-power mean of $|Gs|$ over all $s \in S$.  Clearly, $\orb_t(G,S)$ for $t = 1$ is the average orbit size of the elements of $S$.  Moreover, for $t = -1$, we have
$$\orb_{-1}(G,S) =  \left(\frac{\sum_{s \in S} |Gs|^{-1}}{|S|}\right)^{-1} = \frac{|S|}{\sum_{O \in S/G} |O||O|^{-1}} = \frac{|S|}{|S/G|},$$
and therefore $\orb_{-1}(G,S)$ is the average number of elements of $S$ in each orbit.  
Taking appropriate limits at $t = 0,\pm \infty$, one can define $\orb_t(G,S)$ for all $t \in [-\infty,\infty]$, and then
$$\orb_\infty(G,S) = \max\{|Gs|: s \in S\}$$
is the maximal orbit size of $S$, and
$$\orb_{-\infty}(G,S) = \min\{|Gs|: s \in S\}$$
is the minimal orbit size.   Clearly, then, every $G$-orbit of $S$ has the same size if and only if the function
$\orb_t(G,S)$ is constant with respect to $t$.  One can use calculus to show that, if the function $\orb_t(G,S)$ is not constant, then it is bounded and has positive derivative everywhere.  Moreover, from $|S|$ and the function $\orb_t(G,S)$ for $t \in [0,1]$, one can recover all of the orbit sizes.   Our general philosophy is that the critical region of interest of the function $\orb_t(G,S)$ is the interval $t \in [-1,1]$, with the value at $t = 1$ being the most important.

For any finite $G$-set $S$, we define the {\bf $t$-power diameter $\diam_t(G,S)$ of $S$}  to be 
$$\diam_t(G,S) =  \frac{\orb_{\infty}(G,S)}{\orb_{t}(G,S)} =  \frac{\max\{|Gs| : s \in S\} }{\orb_{t}(G,S)}.$$  In other words, $\diam_t(G,S)$ is the ratio of the maximal $G$-orbit size of $S$ to the $t$-power mean  orbit size of the elements of $S$. The function $\diam_t(G,S)$, if not identically $1$, has negative derivative with respect to $t$ and has limiting values of $1$ and  $\frac{\max\{|Gs|: s \in S\}}{\min\{|Gs|: s \in S\}}$ at $t = \infty$ and $t = -\infty$, respectively.  The $t$-power diameter $\diam_t(G,S)$ represents in an intuitive sense the amount of spread in the  ``kinetic energy'' or ``entropy'' of the elements of $S$ under the action of $G$, where elements with larger orbits, or equivalently with smaller stabilizers, are considered to have more kinetic energy.  The mathematical problem we pose here is the following.

\begin{problem}
Given a group $G$, a finite $G$-set $S$, and $t \in [-\infty,\infty]$, determine the subgroups $H$ of $G$ for which $\diam_t(H,S)$ is largest. 
\end{problem}

 Such subgroups $H$ of $G$ maximize the spread of the ``kinetic energy'' of the elements of $S$ under the induced goup  action.

The following is our main result regarding heptatonic scales in $\ZZ_{12}$.

\begin{theorem}[with James Allen,  Paul Estrada, and Michael McCann]\label{mainconjecture}
For all $t \in [-1,1]$, the subgroups $G$ of $S_{11}$ for which $\diam_t(G,\TT_7)$ is largest are the group $\Gamma$ generated by $\{(1 \ 2),(3 \ 4),(5 \ 6),(8 \ 9),(10 \ 11)\}$, along with its conjugate subgroups.  
\end{theorem}

The theorem can be proved numerically using GAP and SAGE, as follows.  First, note that, for any group $G$ and any finite $G$-set $S$, the number $\diam_t(H,S)$ for any subgroup $H$ of $G$ depends only on the conjugacy class of $H$.  The group  $S_{11}$ has 3094 subgroups up to conjugacy.  (Those that are cyclic have order $1$, $2$, $3$, $4$, $5$, $6$, $7$, $8$, $9$, $10$, $11$, $12$, $14$, $15$, $18$, $20$, $21$, $24$, $28$, or $30$.)  Using GAP and SAGE, one can compute generators for representatives of all 3094 conjugacy classes, and for each of these representatives $G$ one can compute the $G$-orbits of $\TT_7$.  Then the functions $\diam_t(G,\TT_7)$ can be plotted and verified to achieve a maximum for $G = \Gamma$ for all $t$ in the interval $[-1,1]$.  

 The fact that $\diam_t(G,\TT_7)$ attains a maximum on the entire interval $[-1,1]$ for a single conjugacy class of subgroups of $S_{11}$ is in itself a surprising result.  One has
$$\diam_1(\Gamma, \TT_7) \approx 3.5250,$$
$$\diam_0(\Gamma, \TT_7) \approx  4.8324,$$  $$\diam_{-1}(\Gamma, \TT_7) \approx  6.6494.$$
Thus, for example, a consequence of Theorem \ref{mainconjecture} is that, for any subgroup $G$ of $S_{11}$, the maximal $G$-orbit size of $\TT_7$ is  at most $3.5251$  times the average $G$-orbit size of the elements of $\TT_7$, and the maximum ratio possible is obtained precisely by the group $G = \Gamma$ and its conjugates.

Given a subgroup $G$ of $S_{11}$ and a scale $s \in \TT$, we define the {\bf $(G,t)$-musicality of $s$} to be the quantity $$m(G,t,s) = \frac{|Gs|}{\orb_t(G,\TT_{|s|})},$$ which is the size of the $G$-orbit of $s$ relative to the $t$-power mean orbit size of the elements of $\TT_{|s|}$.  Thus, the quantity $$\diam_t(G,\TT_k) = \max\{m(G,t,s): s \in \TT_k\}$$ represents the largest possible $(G,t)$-musicality of any $k$-tonic scale.  Theorem \ref{mainconjecture} says that the heptatonic scales with the largest possible $(G,t)$-musicality for any subgroup $G$ of $S_{11}$ and any $t \in [-1,1]$ occur precisely for the group $G = \Gamma$ and its conjugates.  These scales comprise the unique maximal $\Gamma$-orbit of $\TT_7$ and consist of the 32 heptatonic scales
$$\left\{0, {1 \atop 2}, {3 \atop 4}, {5 \atop 6}, 7, {8 \atop 9}, {10 \atop 11} \right\},$$
or equivalently, starting, say, at C, the 32 scales
$$\left\{\mbox{C}, {\mbox{D}\flat \atop \mbox{D}}, {\mbox{E}\flat \atop \mbox{E}}, {\mbox{F} \atop \mbox{F}\sharp}, \mbox{G}, {\mbox{A}\flat \atop \mbox{A}}, {\mbox{B}\flat \atop \mbox{B}}\right\}$$
listed in Table 1.  This orbit contains the major scale, the harmonic and melodic minor scales, and many other heptatonic scales that figure prominently in Western and Indian classical music.  In fact, all 32 of these scales are among the 72  m\={e}\d{l}akarta ragas of Carnatic (South Indian) classical music standardized by Govindacharya in the 18th century and coincide with the 32 th\=ats of Hindustani classical music popularlized by the system created by Vishnu Narayan Bhatkhande (1860--1936), one of the most influential musicologists in the field of Hindustani classical music in the twentieth century.

\begin{table}
\caption{The 32 scales in the maximal $\Gamma$-orbit of $\TT_7$ (the 32 th\=ats of Hindustani classical music)}
\centering 
\begin{tabular}{l|lllllll} \hline
major, Ionian mode, or Bil\=awal th\=at & C & D & E & F & G & A & B \\ \hline
Mixolydian or Adonai malakh mode, or Khamaj  th\=at & C & D & E & F & G & A & B$\flat$ \\ \hline
harmonic major & C & D & E & F & G & A$\flat$ & B \\ \hline
Mixolydian b6 & C & D & E & F & G & A$\flat$ & B$\flat$ \\ \hline
Lydian mode, or Kalyan th\=at & C & D & E & F$\sharp$ & G & A & B \\ \hline
acoustic, or Lydian dominant & C & D & E & F$\sharp$ & G & A & B$\flat$ \\ \hline
 & C & D & E & F$\sharp$ & G & A$\flat$ & B \\ \hline
minor Lydian & C & D & E & F$\sharp$ & G & A$\flat$ & B$\flat$ \\ \hline
ascending melodic minor  & C & D & E$\flat$ & F & G & A & B \\ \hline
Dorian mode, or K\=afi th\=at & C & D & E$\flat$ & F & G & A & B$\flat$ \\ \hline
harmonic minor & C & D & E$\flat$ & F & G & A$\flat$ & B \\ \hline
natural minor, Aeolian mode, or \=As\=avari th\=at & C & D & E$\flat$ & F & G & A$\flat$ & B$\flat$ \\ \hline
diminished Lydian & C & D & E$\flat$ & F$\sharp$ & G & A & B \\ \hline
Ukrainian Dorian  & C & D & E$\flat$ & F$\sharp$ & G & A & B$\flat$ \\ \hline
Hungarian minor & C & D & E$\flat$ & F$\sharp$ & G & A$\flat$ & B \\ \hline
gypsy & C & D & E$\flat$ & F$\sharp$ & G & A$\flat$ & B$\flat$ \\ \hline
 & C & D$\flat$ & E & F & G & A & B \\ \hline
 & C & D$\flat$ & E & F & G & A & B$\flat$ \\ \hline
double harmonic, or flamenco mode & C & D$\flat$ & E & F & G & A$\flat$ & B \\ \hline
Phrygian dominant & C & D$\flat$ & E & F & G & A$\flat$ & B$\flat$ \\ \hline
M\=arv\=a th\=at & C & D$\flat$ & E & F$\sharp$ & G & A & B \\ \hline
 & C & D$\flat$ & E & F$\sharp$ & G & A & B$\flat$ \\ \hline
P\=urvi th\=at & C & D$\flat$ & E & F$\sharp$ & G & A$\flat$ & B \\ \hline
 & C & D$\flat$ & E & F$\sharp$ & G & A$\flat$ & B$\flat$ \\ \hline
Neapolitan major  & C & D$\flat$ & E$\flat$ & F & G & A & B \\ \hline
Phrygian raised sixth & C & D$\flat$ & E$\flat$ & F & G & A & B$\flat$ \\ \hline
Neapolitan minor  & C & D$\flat$ & E$\flat$ & F & G & A$\flat$ & B \\ \hline
Phrygian mode, or Bhairav th\=at &  C & D$\flat$ & E$\flat$ & F & G & A$\flat$ & B$\flat$ \\ \hline
 & C & D$\flat$ & E$\flat$ & F$\sharp$ & G & A & B \\ \hline
 & C & D$\flat$ & E$\flat$ & F$\sharp$ & G & A & B$\flat$ \\ \hline
Todi th\=at & C & D$\flat$ & E$\flat$ & F$\sharp$ & G & A$\flat$ & B \\ \hline
Bhairavi th\=at, or Pelog (approximate) & C & D$\flat$ & E$\flat$ & F$\sharp$ & G & A$\flat$ & B$\flat$ \\ \hline
\end{tabular}
\end{table}

The remainder of this paper is organized as follows.  In Section 2 we discuss power means and expected values, and in Section 3 we apply Section 2 to the study of the power mean orbit size and diameter of a finite $G$-set.  In Section 4 we apply Section 3 to the study of scales in $\ZZ_{12}$.  In Section 5 we focus on heptatonic scales in particular, while in Section 6 we briefly study hextonic scales, and in Section 7 we study pentatonic scales.  

I would like to thank the two reviewers for their thoughtful and invaluable input on the first draft of this paper.  As one of the reviewers pointed out, it is likely that the methods of this paper can be combined synergistically with other ways of understanding musical scales, such as those developed in \cite{car}, \cite{clo1}, \cite{clo2}, \cite{clo3}, \cite{dou}, \cite{hoo1}, \cite{hoo2}, \cite{tym}, and \cite{zab}.  It is not my intention that the results in this paper be definitive.   My aim is merely to provide yet another perspective on the already  well-developed mathematical theories of musical scales, one that, in my view, has also inspired some new and interesting problems regarding the theory of group actions.  Based on the reviewers comments, I also discuss some ways in which the theory might be amended or developed further.

The research for this paper was conducted  with undergraduate students James Allen and Paul Estrada and MS student Michael McCann at California State University, Channel Islands, in the academic year 2016--17, under the supervision of the author.  The idea for the project began with conversations between the author and undergraduate student Vickie Chen during a semester-long project on group theory in music for a first course in abstract algebra.  It is in those conversations that Vickie and I first came up with the idea of examining the maximal $\Gamma$-orbit of $\TT_7$, an idea that, to our pleasant surprise, eventually led to Theorem \ref{mainconjecture}.

\section{Power means and expected values}

Arithmetic means are generalized by what are known as {\it power means}.  If $S = \{s_1, \ldots, s_r\}$ is a finite set of cardinality $r = |S|$ and $X: S \longrightarrow \RR_{>0}$ a positive real-valued random variable on $S$ (with the uniform distribution on $S$), then, for any nonzero $t \in \RR$, the {\bf $t$-power mean} of $X$ is defined to be the positive real number
$$\MM_t(X) = \MM_t(X(s): s \in S) = \MM_t(x_1, \ldots, x_r )= \left(\frac{\sum_{i = 1}^r x_i^t}{r}\right)^{1/t},$$
where $x_i = X(s_i)$ for all $i$.  Equivalently, the arithmetic mean of $X$ is just $\MM_1(X)$, and one sets
$$\MM_t(X) = \MM_1(X^t)^{1/t}.$$
For $a = 0, \pm \infty$, one defines $$\MM_a(X) = \MM_a(X(s): s \in S) = \MM_a(x_1, \ldots, x_r) = \lim_{t \rightarrow a} \MM_t(x_1, \ldots, x_r).$$
It is well known that
$$\MM_0(x_1,\ldots, x_r) = \left(\prod_{i = 1}^r x_i\right)^{1/r}$$
is the geometric mean of the $x_i$, while
$$\MM_\infty(x_1, \ldots, x_r) = \max(x_1, \ldots, x_r)$$
and
$$\MM_{-\infty}(x_1, \ldots, x_r) = \min(x_1, \ldots, x_r)$$
are the maximum and minimum, respectively, of the $x_i$.

Let $[-\infty,\infty] = \RR \cup \{\infty,-\infty\}$ denote the set of all {\bf  extended real numbers}.  For all $t \in [-\infty,\infty] $ one has
$$\MM_t(cX) = c\MM_t(X)$$ for all $c > 0$ and
$$\MM_t(X^{-1}) = \MM_{-t}(X)^{-1}.$$   It follows that, if $XY = c$, that is, if $Y = c/X$, then $$\MM_t(X) \MM_{-t}(Y) = c$$ for all $t$.


We may generalize the definition of $\MM_t(X)$ by assuming that $S$ is a finite probability space with probability distribution
$\PP: S \longrightarrow [0,1]$, which we assume is nonzero at all elements of $S$, and $X: S \longrightarrow \RR_{>0}$ is a positive real-valued random variable on $S$.  (Previously we implicitly assumed  that $\PP(s) = \frac{1}{|S|}$ for all $s \in S$.)  We may define the {\bf $t$-power expected value} of $X$ to be
$$\EE_t(X) = \EE_t(X(s): s \in S) = \left(\sum_{s \in S} \PP(s)X(s)^t\right)^{1/t}.$$
Equivalently, the expected value of $X$ is just $\EE_1(X)$, and one sets
$$\EE_t(X) = \EE_1(X^t)^{1/t}.$$
For $a = 0, \pm \infty$, one defines  $$\EE_a(X)   =  \lim_{t \rightarrow a} \EE_t(X).$$
One has
$$\EE_0(X) = \prod_{s \in S} X(s)^{\PP(s)},$$
while
$$\EE_\infty(X) = \max X(S)$$
and
$$\EE_{-\infty}(X) = \min X(S).$$


If $X$ is constant, then clearly $\EE_t(X)$ is a constant function of $t$ and one has
$\EE_t(X) = X(s)$ for all $t$ and all $s \in S$.  Conversely, if $\EE_t(X)$ is a constant function of $t$, then $\max X(S) = \min X(S)$, whence $X$ must be constant.

One can show that the function $\EE_t(X)$ of $t$ is differentiable with nonnegative derivative,  and is therefore nondecreasing, with respect to $t$.   Thus, one has
$$\min X(S) \leq \EE_t(X) \leq \max X(S)$$
for all $t$.  Moreover, if $X$ is nonconstant, then $\EE_t(X)$ has positive derivative, and therefore is strictly increasing, with respect to $t$.  In other words, if nonconstant, the function $\EE_t(X)$ of $t$ is a {\bf sigmoid function}, that is, it is a bounded differentiable function from $\RR$ to $\RR$ whose derivative is everywhere positive.  Thus it has horizontal asymptotes, specifically at $y = \max X(S)$ and $y = \min X(S)$ at $\infty$ and $-\infty$, respectively.  Thus, its graph is an ``S-shaped'' curve.
For an explicit example using the uniform probability distribution, see Figure 1.

\begin{figure}[ht!]
\centering
\includegraphics[width=75mm]{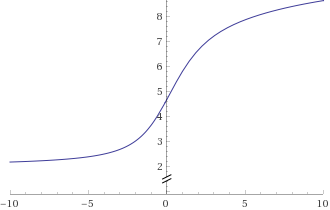}
\caption{Graph of $\MM_t(2,2,7,8,10) = \left(\frac{2^t + 2^t + 7^t + 8^t + 10^t)}{5}\right)^{1/t}$ on $[-10,10]$ \label{har5}}
\end{figure}

\section{Power mean orbit size of a finite $G$-set}

Throughout this section, $G$ denotes a group and $S$ a finite $G$-set.  One may readily generalize all of what follows to the situation where $S$ is also assumed to be a probability space; in that case, one simply replaces all $t$-power means with $t$-power expected values.  

For any $t \in [-\infty,\infty]$, we define the {\bf $t$-power mean orbit size of the elements of $S$} to be
$$\orb_t(G,S) = \MM_t(|Gs|: s \in S),$$
which, for $t \neq 0,\pm \infty$ is given by $$\orb_t(G,S) = \left(\frac{\sum_{s \in S} |Gs|^t}{|S|}\right)^{1/t} = \left(\frac{\sum_{O \in S/G} |O|^{t+1}}{|S|}\right)^{1/t} .$$  
As observed in the introduction, $\orb_1(G,S)$  is the average orbit size of the elements of $S$, and  $\orb_{-1}(G,S) = \frac{|S|}{|S/G|}$ is the average number of elements of $S$ in each orbit, while $\orb_\infty(G,s) = \max\{|Gs|: s \in S\}$
is the maximal orbit size of $S$ and $\orb_{-\infty}(G,s) = \min\{|Gs|: s \in S\}$
is the minimal orbit size.  

For any $s \in S$, we define the {\bf $t$-power relative size of $s$} to be
$$|s|_{G,S,t} = \frac{|Gs|}{\orb_t(G,S)}.$$   The $t$-power relative size of $s$ is the size of the orbit of $s$ relative to (or normalized with respect to) the $t$-power mean orbit size of the elements of $S$.  The $t$-power mean of the orbit sizes of the elements of $S$ is equal to $$\MM_t(|Gs|: s\in S) = \orb_t(G,S)$$ of $S$, while  the $t$-power mean  of the $t$-power relative sizes of the elements of $S$ is  equal to $1$:
$$\MM_t(|s|_{G,S,t}: s \in S) = 1.$$ 
Because of this normalization property, if $H$ and $K$ are subgroups of $G$, then it makes sense to compare the values of $|s|_{H,S,t}$ and $|s|_{K,S,t}$ with each other.

The {\bf $t$-power diameter $\diam_t(G,S)$ of $S$}, as defined in the introduction, is equivalently the maximal $t$-power relative size of an element of $S$, that is, one has
$$\diam_t(G,S)  = \max\{|s|_{G,S,t} : s \in S\} = \frac{\max\{|Gs| : s \in S\} }{\orb_{t}(G,S)} =  \frac{\orb_{\infty}(G,S)}{\orb_{t}(G,S)}.$$  The function $\diam_t(G,S)$, if not identically $1$, has negative derivative with respect to $t$ and has limiting values of $1$ and  $\frac{\max\{|Gs|: s \in S\}}{\min\{|Gs|: s \in S\}}$ at $t = \infty$ and $t = -\infty$, respectively. 

For example, Figure 2 provides the graph of $\orb_t(G,S)$ and $\diam_t(G,S)$  for any $G$-set $S$ with orbit sizes $2$, $2$, $7$, $8$, and $10$.

\begin{figure}[ht!]
\centering
\includegraphics[width=120mm]{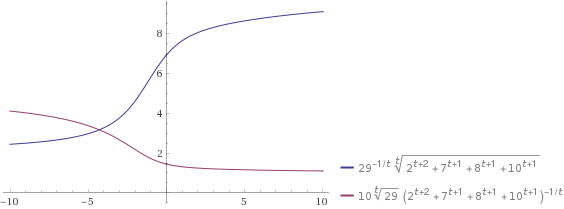}
\caption{Graph of $\orb_t(G,S)= \left(\frac{2^{t+1} + 2^{t+1} + 7^{t+1} + 8^{t+1} + 10^{t+1}}{29}\right)^{1/t}$ and $\diam_t(G,S) = 10\left(\frac{2^{t+1} + 2^{t+1} + 7^{t+1} + 8^{t+1} + 10^{t+1}}{29}\right)^{-1/t}$ on $[-10,10]$ for any $G$-set $S$ with orbit sizes  $2$, $2$, $7$, $8$, and $10$.  \label{har5}}
\end{figure}

\section{Musicality of scales and groups}

Throughout this section, $G$ denotes a subgroup of $S_{11}$, which acts on the set $\TT$ of all $2^{11}$ possible tonic scales, as well as on the subsets 
$\TT_k$ of all ${11 \choose k-1}$ possible $k$-tonic scales, for all $k = 1, 2, \ldots, 12$, as described in the introduction, and $t$ denotes a number or variable with values in the extended reals $[-\infty,\infty] $.

For any $s \in \TT$, that is, for any tonic scale $s$, we define the {\bf  $(G,t)$-musicality} of $s$, or of the orbit $Gs$, to be  $$m(G,t,s) = |s|_{G, \TT_{|s|},t} = \frac{|Gs|}{\orb_t(G,\TT_{|s|})},$$ which is the $t$-power relative size of $s$ in $\TT_{|s|}$ (not in $\TT$).  Musicality defines a natural function 
$$m:  \operatorname{Subgp}(S_{11}) \times [-\infty,\infty] \times  \TT \longrightarrow [1,\infty),$$ where $\operatorname{Subgp}(S_{11})$ denotes the lattice of all subgroups of $S_{11}$. For a fixed $G$, $t$, and $k$, the $t$-power mean $\MM_t(m(G,t,s): s \in \TT_k)$ of $m(G,t,s)$ over all $k$-tonic scales $s$  is equal to $1$.  The $(G,t)$-musicality $m(G,t,s)$ of a $k$-tonic scale $s$ is directly proportional to the size of its $G$-orbit.  The constant of proportionality depends on $t$ and is defined natually in such a way that one can meaningfully compare the values for various subgroups $G$ of $S_{11}$ for a fixed $s$, or compare the values for various scales $s$ for a fixed group $G$.  

Observe that
$$\diam_t(G, \TT_k) = \frac{\max\{|Gs|: s \in S\}}{\orb_t(G,\TT_{k})} = \max\{m(G,t,s): s \in \TT_k\}.$$  
Thus $\diam_t(G, \TT_k)$ is the largest possible $(G,t)$-musicality of a scale in $\TT_k$, or equivalently it is the $(G,t)$-musicality of any maximal $G$-orbit of $\TT_k$.  If we let $\TT_k//G$ denote the set of all maximal $G$-orbits of $\TT_k$, then the union $\TT_{k,G}  = \bigcup (\TT_k//G)$ is the set of all scales in $\TT_k$ that have the largest possible $(G,t)$-musicality (namely, $\diam_t(G, \TT_k)$) for any $t$. We call the scales in the set $\TT_{k,G}$  the {\bf $k$-tonic scales of $G$}.    Our philosophy is that the scales in the set $\TT_{k,G}$ should be regarded as the optimally musical $k$-tonic scales relative to $G$.

Let $G$ be a subgroup of $S_n$.  We define a {\bf signature of $G$ (in $S_n$)} to be a list  $(n_1,n_2,\ldots, n_k)$ of the $G$-orbit sizes of $\{\{1\},\{2\},\ldots,\{n\}\}$, listed in any particular order.    In particular, one has $n = n_1 + n_2 + \cdots + n_k$ for any signature $(n_1,n_2,\ldots, n_k)$ of $G$ in $S_n$.  We say that $G$ {\bf acts without crossings (in $S_n$)},  if the $G$-orbits of the $G$-set $\{\{1\},\{2\},\ldots,\{n\}\}$ are all of the form $\{\{a\},\{a+1\},\ldots,\{a+k\}\}$, where the addition here is ordinary addition of positive integers, not addition modulo $n$.   (In the case at hand, $n = 11$, and $0$ is omitted from the discussion because we have chosen to leave $0$ fixed by $S_{11}$.)  It is clear that every subgroup of $S_n$ is conjugate to a subgroup that acts without crossings.  In fact, the conjugates of $G$ that act without crossings in $S_n$ are in one-to-one correspondence with the signatures of $G$ in $S_n$.  If $G$ acts without crossings in $S_n$, we define {\bf the signature of $G$ (in $S_n$)} to be the list $(n_1,n_2,\ldots, n_k)$ of the orbit sizes of $\{\{1\},\{2\},\ldots,\{n\}\}$, listed in order so that the orbits are $\{\{1\},\{2\},\ldots,\{n_1\}\}$, $\{\{n_1+1\},\{n_1+2\},\ldots,\{n_1+n_2\}\}$, etc.

Let $n_1, n_2, \ldots, n_d$ be a sequence of positive integers whose sum is $11$.  Then the group $S_{n_1} \times S_{n_2} \times \cdots \times S_{n_d}$ naturally embeds into $S_{11}$ in the following way.  The first factor $S_{n_1}$ acts on the first $n_1$ numbers, $1,2,\ldots, n_1$.  The next factor $S_{n_2}$ acts on the next $n_2$ numbers, $n_1+1, n_1+2, \ldots, n_1+n_2$.  And so onward to the last factor $S_{n_d}$, which acts on the last $n_d$ numbers, $n_1+\cdots+n_{d-1}+1$ to $11$.  We denote the image of the embedding $$\Phi: S_{n_1} \times S_{n_2} \times \cdots \times S_{n_d} \longrightarrow S_{11}$$ described above by $S_{n_1,n_2, \ldots, n_d}$.   We say that a {\bf twelve tone group of signature at most $(n_1, n_2, \ldots n_d)$} is a subgroup of $S_{n_1,n_2, \ldots, n_d}$ of the form $\Phi(G_1 \times G_2 \times \cdots \times G_d)$, where $G_i$ is a subgroup of $S_{n_i}$ for all $i$. Note that the group $S_{n_1,n_2, \ldots, n_d} \cong S_{n_2} \times \cdots \times S_{n_d}$ is the largest twelve tone group of signature at most $(n_1, n_2, \ldots n_d)$ in the sense that it contains every twelve tone group of signature at most $(n_1, n_2, \ldots n_d)$.   For this reason we call it {\bf the maximal twelve tone group of signature $(n_1, n_2, \ldots n_d)$}.   For example, the group $S_{11}$ the maximal twelve tone group of signature $(11)$, and therefore every subgroup of $S_{11}$ is a twelve tone group of signature $(11)$.  Note that all twelve tone groups act without crossings, and the signature of the maximal twelve tone group of signature $(n_1, n_2, \ldots n_d)$ is $(n_1, n_2, \ldots n_d)$.  

Of course, there are other natural embeddings of $S_{n_1} \times S_{n_2} \times \cdots \times S_{n_d}$ in $S_{11}$.  For example, $S_4 \times S_7$ can be embedded in $S_{11}$ by allowing the first factor to act, say, on $\{2,4,7,8\}$ and the second factor on $\{1,3,5,6,9,10,11\}$.  Such an embedding does not yield a twelve tone group with signature $(4,7)$.   In loose terminology, twelve tone groups do not allow the factors to act in a ways that are ``intertwined'': no ``crossing'' is allowed.  Philosophically, this restriction can be motivated as follows.  The chromatic scale has a linear ordering, and our goal is to understand how various scales may be transformed from one to the other.  The most obvious and most common way in which this is done is by changing various notes of the scale by applying operations $\natural, \sharp, \flat$, which can be modeled by permuting neighboring notes.  The notion of a twelve tone group is meant to capture this notion of locality.  

Even without these locality restrictions, our mathematical analysis of twelve tone groups will apply equally well to any of the embeddings of $S_{n_1} \times S_{n_2} \times \cdots \times S_{n_d}$ in $S_{11}$ that allow crossings, because we can simply relabel the numbers $0$ through $11$ so that  there are no crossings.  We say that a {\bf maximal twelve tone group of signature $(n_1, n_2, \ldots, n_d)$ with or without crossings} is a subgroup of $S_{11}$ that is the result of applying some (inner) automorphism of $S_{11}$ to the maximal twelve tone group $S_{n_1, n_2,\ldots,n_d}$.  We note the following proposition, whose proof is elementary.

\begin{proposition}
There are $2^{10} = 1024$ possible signatures, hence $1024$ maximal twelve tone groups, corresponding to $10$ independent choices of whether or not to separate $i$ from $i+1$, for $i = 1,\ldots,10$.  There are $678570$ maximal twelve tone groups with or without crossings, corresponding to the $678570$ possible partitions of an eleven element set.  Up to isomorphism, there are $p(11) = 56$ maximal twelve tone groups (or maximal twelve tone groups with or without crossings), corresponding to the $56$ possible partitions of the number $11$.
\end{proposition}

\section{Heptatonic scales}

In this section we are primarily interested in the action of subgroups of $S_{11}$ on the set $\TT_7$ of all heptatonic (tonic) scales, which has 462 elements.

The following proposition gives a formula for $\diam_t(G,\TT_7)$ for any maximal twelve tone group $G$.   {\bf Multisets} are a generalization of sets where, as with tuples, repetition is allowed, but, as with sets, order doesn't matter.

\begin{proposition}\label{prop:orb}
Let $G$ be a maximal twelve tone subgroup of $S_{11}$ of signature $(n_1, n_2, \ldots, n_d)$ with or without crossings.  The multiset of $G$-orbit sizes of $\TT_7$ is the multiset
$$\left\{{n_1 \choose k_1} {n_2 \choose k_2} \cdots  {n_d \choose k_d}: k_1, k_2, \ldots, k_d \in \ZZ_{> 0}, k_1+k_2 + \cdots + k_d = 6, \forall i \ {k_i \leq n_i} \right\}$$
of positive integers.
Therefore, one has
 $$\orb_t(G,\TT_7)  =  \left(\frac{1}{462}  \sum_{{k_1+k_2 + \cdots + k_d = 6} \atop {k_i \leq n_i}}  \left( {n_1 \choose k_1}  {n_2 \choose k_2} \cdots  {n_d \choose k_d}  \right)^{t+1 } \right)^{1/t}$$
for all $t \neq 0, \pm \infty$, the maximal $G$-orbits of $\TT_7$ have size
$$\orb_\infty(G,\TT_7) = \max_{{k_1+k_2 + \cdots + k_d = 6} \atop {k_i \leq n_i}} {n_1 \choose k_1} {n_2 \choose k_2} \cdots  {n_d \choose k_d},$$
and one has
$$\diam_t(G,\TT_7) = \frac{\orb_\infty(G,\TT_7) }{\orb_t(G,\TT_7) }.$$
One also has
$$462 = {11 \choose 6} =  \sum_{{k_1+k_2 + \cdots + k_d = 6} \atop {k_i \leq n_i}} {n_1 \choose k_1} {n_2 \choose k_2} \cdots  {n_d \choose k_d}$$
and the number of $G$-orbits of $\TT_7$ is equal to 
$$\sum_{{k_1+k_2 + \cdots + k_d = 6} \atop {k_i \leq n_i}} 1.$$
\end{proposition}

\begin{proof}
The proof is elementary.
\end{proof}

Clearly, this proposition generalizes to $k$-tonic scales by replacing $\TT_7$ everywhere in the proposition with $\TT_k$, replacing the number $6$ everywhere with the number $k-1$, and replacing the number 462 with ${11 \choose k-1}$.  (It even works for an $N$-note chromatic scale by replacing $11$ everywhere with $N-1$.)  It also generalizes to $k$-atonic scales.

Using the proposition, one can compute $\diam_t(G,\TT_7)$ for all 56 maximal twelve tone groups $G$, say, for the critical values $t = 1,0,-1$.  These values are listed in Table 2 in descending order of $\diam_1(G,\TT_7)$.

\begin{table} 
\caption{Maximal twelve tone groups}
\scriptsize
\centering 
{\begin{tabular}{l||l|l|l|l|l}
Signature &  Maximal orbits & $\#$ orbits & $\diam_1(G,\TT_7)$ & $\diam_0(G,\TT_7)$ &  $\diam_{-1}(G,\TT_7)$  \\\hline \hline
$(2, 2, 2, 2, 2, 1)$ & 1 of size 32 & 96 & 3.5250 & 4.8324 & 6.6494 \\ \hline
$(3, 2, 2, 2, 2)$ & 1 of size 48 & 61 & 3.0689 & 4.3060 & 6.3377 \\ \hline
$(2, 2, 2, 2, 1, 1, 1)$ & 3 of size 16 & 131 & 2.7603 & 3.5264 & 4.5368 \\ \hline
$(4, 2, 2, 2, 1)$ & 1 of size 48 & 48 & 2.6679 & 3.4917 & 4.9870 \\ \hline
$(3, 2, 2, 2, 1, 1)$ & 3 of size 24 & 83 & 2.4115 & 3.1423 & 4.3117 \\ \hline
$(5, 2, 2, 2)$ & 1 of size 80 & 26 & 2.3864 & 3.1193 & 4.5022 \\ \hline
$(4, 3, 2, 2)$ & 1 of size 72 & 31 & 2.3203 & 3.1114 & 4.8312 \\ \hline
$(2, 2, 2, 1, 1, 1, 1, 1)$ & 10 of size 8 & 179 & 2.1513 & 2.5733 & 3.0996 \\ \hline
$(3, 3, 2, 2, 1)$ & 3 of size 36 & 53 & 2.1027 & 2.8000 & 4.1299 \\ \hline
$(4, 2, 2, 1, 1, 1)$ & 3 of size 24 & 65 & 2.0921 & 2.5480 & 3.3766 \\ \hline
$(4, 4, 2, 1)$ & 1 of size 72 & 25 & 2.0182 & 2.5229 & 3.8961 \\ \hline
$(6, 2, 2, 1)$ & 1 of size 80 & 18 & 1.9833 & 2.3882 & 3.1169 \\ \hline
$(3, 2, 2, 1, 1, 1, 1)$ & 10 of size 12 & 113 & 1.8844 & 2.2930 & 2.9351 \\ \hline
$(5, 2, 2, 1, 1)$ & 3 of size 40 & 35 & 1.8788 & 2.2763 & 3.0303 \\ \hline
$(3, 3, 3, 2)$ & 3 of size 54 & 34 & 1.8293 & 2.4951 & 3.9740 \\ \hline
$(4, 3, 2, 1, 1)$ & 3 of size 36 & 42 & 1.8261 & 2.2705 & 3.2727 \\ \hline
$(7, 2, 2)$ & 1 of size 140 & 9 & 1.8033 & 2.1611 & 2.7273 \\ \hline
$(5, 4, 2)$ & 1 of size 120 & 14 & 1.8032 & 2.2539 & 3.6364 \\ \hline
$(4, 4, 3)$ & 1 of size 108 & 16 & 1.7533 & 2.2482 & 3.7403 \\ \hline
$(6, 3, 2)$ & 1 of size 120 & 12 & 1.7228 & 2.1281 & 3.1169 \\ \hline
$(2, 2, 1, 1, 1, 1, 1, 1, 1)$ & 35 of size 4 & 245 & 1.6709 & 1.8779 & 2.1212 \\ \hline
$(3, 3, 2, 1, 1, 1)$ & 10 of size 18 & 72 & 1.6480 & 2.0433 & 2.8052 \\ \hline
$(5, 3, 2, 1)$ & 3 of size 60 & 23 & 1.6364 & 2.0284 & 2.9870 \\ \hline
$(4, 2, 1, 1, 1, 1, 1)$ & 10 of size 12 & 88 & 1.6325 & 1.8594 & 2.2857 \\ \hline
$(4, 3, 3, 1)$ & 3 of size 54 & 27 & 1.5907 & 2.0232 & 3.1558 \\ \hline
$(4, 4, 1, 1, 1)$ & 3 of size 36 & 34 & 1.5849 & 1.8411 & 2.6494 \\ \hline
$(6, 2, 1, 1, 1)$ & 3 of size 40 & 24 & 1.5577 & 1.7428 & 2.0779 \\ \hline
$(6, 4, 1)$ & 1 of size 120 & 10 & 1.4994 & 1.7256 & 2.5974 \\ \hline
$(5, 2, 1, 1, 1, 1)$ & 10 of size 20 & 47 & 1.4704 & 1.6611 & 2.0346 \\ \hline
$(8, 2, 1)$ & 1 of size 140 & 6 & 1.4667 & 1.6141 & 1.8182 \\ \hline
$(3, 2, 1, 1, 1, 1, 1, 1)$ & 35 of size 6 & 154 & 1.4667 & 1.6733 & 2.0000 \\ \hline
$(3, 3, 3, 1, 1)$ & 10 of size 27 & 46 & 1.4388 & 1.8207 & 2.6883 \\ \hline
$(4, 3, 1, 1, 1, 1)$ & 10 of size 18 & 57 & 1.4289 & 1.6569 & 2.2208 \\ \hline
$(7, 2, 1, 1)$ & 3 of size 70 & 12 & 1.4224 & 1.5770 & 1.8182 \\ \hline
$(5, 4, 1, 1)$ & 3 of size 60 & 19 & 1.4220 & 1.6448 & 2.4675 \\ \hline
$(5, 3, 3)$ & 3 of size 90 & 15 & 1.4218 & 1.8074 & 2.9221 \\ \hline
$(7, 4)$ & 1 of size 210 & 5 & 1.3618 & 1.5615 & 2.2727 \\ \hline
$(6, 3, 1, 1)$ & 3 of size 60 & 16 & 1.3583 & 1.5529 & 2.0779 \\ \hline
$(9, 2)$ & 1 of size 252 & 3 & 1.3469 & 1.4752 & 1.6364 \\ \hline
$(6, 5)$ & 1 of size 200 & 6 & 1.3379 & 1.5416 & 2.5974 \\ \hline
$(2, 1, 1, 1, 1, 1, 1, 1, 1, 1)$ & 126 of size 2 & 336 & 1.2941 & 1.3704 & 1.4545 \\ \hline
$(3, 3, 1, 1, 1, 1, 1)$ & 35 of size 9 & 98 & 1.2857 & 1.4911 & 1.9091 \\ \hline
$(5, 3, 1, 1, 1)$ & 10 of size 30 & 31 & 1.2848 & 1.4802 & 2.0130 \\ \hline
$(5, 5, 1)$ & 3 of size 100 & 11 & 1.2727 & 1.4694 & 2.3810 \\ \hline
$(8, 3)$ & 1 of size 210 & 4 & 1.2725 & 1.4383 & 1.8182 \\ \hline
$(4, 1, 1, 1, 1, 1, 1, 1)$ & 35 of size 6 & 119 & 1.2692 & 1.3569 & 1.5455 \\ \hline
$(7, 3, 1)$ & 3 of size 105 & 8 & 1.2375 & 1.4053 & 1.8182 \\ \hline
$(6, 1, 1, 1, 1, 1)$ & 10 of size 20 & 32 & 1.2171 & 1.2718 & 1.3853 \\ \hline
$(8, 1, 1, 1)$ & 3 of size 70 & 8 & 1.1538 & 1.1779 & 1.2121 \\ \hline
$(5, 1, 1, 1, 1, 1, 1)$ & 35 of size 10 & 63 & 1.1458 & 1.2122 & 1.3636 \\ \hline
$(3, 1, 1, 1, 1, 1, 1, 1, 1)$ & 126 of size 3 & 210 & 1.1379 & 1.2211 & 1.3636 \\ \hline
$(7, 1, 1, 1, 1)$ & 10 of size 35 & 16 & 1.1149 & 1.1508 & 1.2121 \\ \hline
$(10, 1)$ & 1 of size 252 & 2 & 1.0820 & 1.0864 & 1.0909 \\ \hline
$(9, 1, 1)$ & 3 of size 126 & 4 & 1.0645 & 1.0765 & 1.0909 \\ \hline
$(11)$ & 1 of size 462 & 462 & 1 & 1 & 1 \\ \hline
$(1, 1, 1, 1, 1, 1, 1, 1, 1, 1, 1)$ & 462 of size 1 & 462 & 1 & 1 & 1 \\ \hline
\end{tabular}}
\label{}
\end{table}

As the most important example, consider the maximal twelve tone group $\Gamma = S_{2,2,2,1,2,2}$ with signature $(2,2,2,1,2,2)$.  Equivalently, $\Gamma$ is the subgroup  of $S_{11}$ generated by the set $\{(1 \ 2), (3\ 4), (5\  6),(8\ 9),(10\  11)\}$, and it is isomorphic to $S_{2} \times S_{2} \times S_{2} \times S_1 \times S_2 \times S_2 \cong \ZZ_2^5$ and has order $32$.  The orbits in $\TT_7$ therefore have 1, 2, 4, 8, 16, or 32 elements. Using Proposition \ref{prop:orb} we find that $\TT_7$ under the action of $\Gamma$ has:
\begin{enumerate}
\item $1$ orbit of size $32$,
\item $5$ orbits of size $16$,
\item $20$ orbits of size $8$,
\item $30$ orbits of size $4$,
\item $30$ orbits of size $2$,
\item $10$ orbits of size $1$,
\end{enumerate}
for a total of $96$ orbits.  One therefore has
$$\orb_t(\Gamma,\TT_7)  = \left(\frac{32^{t+1}+ 5 \cdot 16^{t+1}+ 20 \cdot 8^{t+1}+30 \cdot 4^{t+1}+ 30 \cdot 2^{t+1}+10 \cdot 1^{t+1}}{462}\right)^{1/t}$$
and
$$\diam_t(\Gamma,\TT_7)  = 32\left( \frac{ 462}{32^{t+1}+ 5 \cdot 16^{t+1}+ 20 \cdot 8^{t+1}+30 \cdot 4^{t+1}+ 30 \cdot 2^{t+1}+10 \cdot 1^{t+1} }\right)^{1/t}.$$
In particular, one has
$$\diam_1(\Gamma,\TT_7) =32 \cdot  \frac{462}{4194} \approx 3.5250$$
and
$$\diam_{-1}(\Gamma,\TT_7) = 32 \cdot \frac{96}{462} \approx 6.6494.$$
The heptatonic scales of $\Gamma$ comprise the unique maximal $\Gamma$-orbit of $\TT_7$, which are the 32 scales
$$\left\{0, {1 \atop 2}, {3 \atop 4}, {5 \atop 6}, 7, {8 \atop 9}, {10 \atop 11} \right\},$$
or
$$\{0, 2^\dagger, 4^\dagger, 6^\dagger, 7, 9^\dagger,11^\dagger\},$$
where each of the $\dagger$'s is either a $\natural$ ($+0$) or a $\flat$ ($-1$).   This unique maximal $\Gamma$-orbit consists of the 32  th\=ats of Hindustani classical music popularized by Bhatkhande.   

One may also consider the maximal twelve tone group $\Gamma_{-}$ of signature $(2,2,1,2,2,2)$, an action that globally fixes $5$ instead of $7$.  Since changing the order of the numbers in the signature does not affect any of the relevant values, the values above calculated for $\Gamma_{-}$ are the same as for $\Gamma$.  The maximal $\Gamma_{-}$-orbit of $\TT_7$ contains the 32 scales
$$\left\{0, {1 \atop 2}, {3 \atop 4}, 5, {6 \atop 7}, {8 \atop 9}, {10 \atop 11} \right\},$$
or
$$\{0, 2^\dagger, 4^\dagger,5,7^\dagger,9^\dagger,11^\dagger\},$$
where each of the $\dagger$'s is either a $\natural$ ($+0$) or a $\flat$ ($-1$). 
In particular, the most ``sharp'' of these 32 scales is precisely the major scale.  The intersection of the maximal $\Gamma$-orbit and the maximal $\Gamma_{-}$-orbit consists of the 16 scales
$$\left\{0, {1 \atop 2}, {3 \atop 4}, 5, 7, {8 \atop 9}, {10 \atop 11} \right\}.$$



We may also consider the compositum ${\Gamma_1} = \Gamma \Gamma_{-}$ of $\Gamma$ and $\Gamma_{-}$, which is the maximal twelve tone group of signature $(2,2,3,2,2)$.   There are:
\begin{enumerate}
\item $1$ orbit of size $48$,
\item $4$ orbits of size $24$,
\item $12$ orbits of size $12$,
\item $4$ orbits of size $8$,
\item $12$ orbits of size $6$,
\item $6$ orbits of size $4$,
\item $6$ orbits of size $3$,
\item $12$ orbits of size $2$,
\item $4$ orbits of size $1$,
\end{enumerate}
for a total of $61$ orbits.   The maximal $\Gamma_1$-orbit is just the union of the maximal $\Gamma$-orbit and the maximal $\Gamma_{-}$-orbit.  Consistent with inclusion-exclusion, one has $48 = 32+32-16$.  One has
$$\diam_1(\Gamma_1,\TT_7) = 48 \cdot \frac{462}{7226} \approx 3.0689$$
and
$$\diam_{-1}(\Gamma_1,\TT_7) = 48 \cdot \frac{ 61}{462} \approx 6.3377.$$
Note that
$$\frac{\diam_1(\Gamma,\TT_7)}{\diam_1(\Gamma_1,\TT_7)} = \frac{32/4194}{48/7226}   \approx \frac{3.5250}{3.0689} \approx 1.1486$$
and
$$\frac{\diam_{-1}(\Gamma,\TT_7)}{\diam_{-1}(\Gamma_1,\TT_7)} = \frac{32 \cdot 96}{48 \cdot 61} \approx \frac{6.6494}{6.3377}  \approx   1.0492.$$

Let us now consider the maximal twelve tone group $\Delta$ of signature $(4,2,1,4)$.  This group is also of theoretical and historical importance in Indian classical  music, as its unique maximal orbit consists precisely of the $72 = 1 \cdot 6 \cdot 2 \cdot 1\cdot 6$ m\={e}\d{l}akarta ragas, which are built as follows:
$$\{0\} \cup \{\mbox{two of } 1,2,3,4\} \cup  \{\mbox{one of } 5,6\} \cup \{7\} \cup  \{\mbox{two of } 8,9,10,11\}.$$
Here we have:
\begin{enumerate}
\item $1$ orbit of size $72$,
\item $2$ orbits of size $48$,
\item $1$ orbit of size $36$,
\item $2$ orbits of size $32$,
\item $4$ orbits of size $24$,
\item $3$ orbits of size $16$,
\item $2$ orbits of size $8$,
\item $2$ orbits of size $6$,
\item $4$ orbits of size $4$,
\item $2$ orbits of size $2$,
\item $2$ orbits of size $1$,
\end{enumerate}
for a total of $25$ orbits.  We then have
$$\diam_1(\Delta,\TT_7) = 72 \cdot \frac{462}{16482} \approx 2.0182$$
and
$$\diam_{-1}(\Delta_1,\TT_7) = 72 \cdot \frac{25}{462} \approx 3.8961$$

Note that, of the 72 m\={e}\d{l}akarta ragas comprising the maximal $\Delta$-orbit, 32 are scales in the maximal $\Gamma_1$-orbit consisting of 48 scales, but only 16 are scales in the maximal $\Gamma_{-}$-orbit consisting of 32 scales.

One may also consider the maximal twelve tone group $\Delta_{-}$ of signature $(4,1,2,4)$, an action that globally fixes $5$ instead of $7$, as does $\Delta$.  Since changing the order of the numbers in the signature does not affect any of the relevant values, the values above calculated for $\Delta_{-}$ are the same as for $\Delta$.  The largest $\Delta_{-}$-orbit of $\TT_7$ contains the 72 scales
$$\{0\} \cup \{\mbox{two of } 1,2,3,4\} \cup  \{5\} \cup   \{\mbox{one of } 6,7\} \cup \{\mbox{two of } 8,9,10,11\}.$$   The intersection of the maximal $\Delta$-orbit and the maximal $\Delta_{-}$-orbit consists of the 36 scales
$$\{0\} \cup \{\mbox{two of } 1,2,3,4\} \cup  \{5,7\} \cup \{\mbox{two of } 8,9,10,11\}.$$

We may also consider the maximal twelve tone group $\Delta_1 = \Gamma_{-}\Delta = \Gamma_1 \Delta$ of signature $(4,3,4)$.  Its largest orbit consists of $108 = 6 \cdot 3 \cdot 6  =72+36$ possible scales, which is the union of the  maximal $\Delta$-orbit and the maximal $\Delta_{-}$-orbit.  Consistent with inclusion-exclusion, one has $108 = 72+72-36$.  The 108 scales are built as follows:
$$\{0\} \cup \{\mbox{two of } 1,2,3,4\} \cup  \{\mbox{two of } 5,6,7\} \cup  \{\mbox{two of } 8,9,10,11\}.$$
These  108 scales include all 72 m\={e}\d{l}akarta ragas in the maximal $\Delta$-orbit and all 32 scales from the maximal $\Gamma_{-}$-orbit, along with 20 others.  Indeed, by inclusion-exclusion, the combined total of scales in the maximal  $\Delta$-orbit and in the maximal $\Gamma_{-}$-orbit is only $88 = 72+32-16$.    Similarly, the 108 scales include all 72 m\={e}\d{l}akarta ragas and all 48 scales from the maximal $\Gamma_1$-orbit, along with the same 20 ``new'' scales.  Indeed, by inclusion-exclusion, the combined total of scales in the maximal  $\Delta$-orbit and in the maximal $\Gamma_1$-orbit is $88 = 72+48-32$.   The 20 new scales are as follows:
$$\{0\} \cup \{1,2\} \cup \{5,6\} \cup \{\mbox{two of } 8,9,10,11\},$$
$$\{0\} \cup \{3,4\} \cup \{5,6\} \cup \{\mbox{two of } 8,9,10,11\},$$
$$\{0\} \cup  \{\mbox{two of } 1,2,3,4\} \cup \{5,6\} \cup \{8,9\},$$
$$\{0\} \cup  \{\mbox{two of } 1,2,3,4\} \cup \{5,6\} \cup \{10,11\}.$$
At first sight this appears to be $24 = 6 \cdot 4$ scales; however, 4 scales are repeated twice, namely, the 4 scales
$$\{0\} \cup (\{1,2\} \mbox{ or } \{3,4\}) \cup \{5,6\} \cup (\{8,9\} \mbox{ or } \{10,11\}).$$

For the group $\Delta_1$ we have:
\begin{enumerate}
\item $1$ orbit of size $108$,
\item $2$ orbits of size $72$,
\item $2$ orbits of size $48$,
\item $2$ orbits of size $24$,
\item $1$ orbit of size $16$,
\item $2$ orbits of size $12$,
\item $2$ orbits of size $6$,
\item $2$ orbits of size $4$,
\item $2$ orbits of size $3$,
\end{enumerate}
for a total of $16$ orbits.  Here we have
$$\diam_1(\Delta_1,\TT_7) = \frac{108 \cdot 462}{28458} \approx 1.7533$$
and
$$\diam_{-1}(\Delta_1,\TT_7)   = \frac{108 \cdot 16}{462} \approx 3.7403.$$
Thus, we see that
$$\frac{\diam_1(\Delta,\TT_7)}{\diam_1(\Delta_1,\TT_7)} = \frac{72/16482}{108/28458} \approx \frac{2.0182}{1.7533}  \approx 1.1511  $$
and
$$\frac{\diam_{-1}(\Delta_1,\TT_7)}{\diam_{-1}1\Delta_1,\TT_7)} = \frac{72 \cdot 25}{108 \cdot 16}   \approx \frac{3.8961}{3.7403} \approx 1.0417.$$
Thus, the passing from $\Delta$ to $\Delta_1$ decreases ``musicality'' of the scales in the maximal orbit in a manner that is comparable to passing from $\Gamma$ to $\Gamma_1$.

The lattice diagram for the maximal twelve tone groups $\Gamma, \Gamma_{-}, \Gamma_1, \Delta, \Delta_{-}, \Delta_1$ that we have discussed thus far, along with the groups $\Gamma_0 = \Gamma \cap \Gamma_-$ and $\Delta_0 = \Delta \cap \Delta_-$, are as follows.
\begin{eqnarray*}
\SelectTips{cm}{11}\xymatrix{  &  {\Delta_1}  & \\
 {\Delta} \ar@{-}[ru] & {\Gamma_1} \ar@{-}[u] &  {\Delta_{-}} \ar@{-}[lu]  \\
 {\Gamma} \ar@{-}[ru] \ar@{-}[u]  &  {\Delta_0} \ar@{-}[ul] \ar@{-}[ur] & {\Gamma_{-}} \ar@{-}[lu] \ar@{-}[u] \\
 &  {\Gamma_0} \ar@{-}[ru] \ar@{-}[lu]  \ar@{-}[u] &}
\end{eqnarray*}
The values of $\diam_1(G,\TT_7)$  for these eight groups are (approximately) as follows.
\begin{eqnarray*}
\SelectTips{cm}{11}\xymatrix{  &  {1.7533}  & \\
 {2.0182} \ar@{-}[ru] & {3.0689} \ar@{-}[u] &  {2.0182} \ar@{-}[lu]  \\
 {3.5250} \ar@{-}[ru] \ar@{-}[u]  & {1.5849}  \ar@{-}[ru] \ar@{-}[lu]& {3.5250} \ar@{-}[lu] \ar@{-}[u] \\
 &  {2.7603} \ar@{-}[ru] \ar@{-}[lu]  \ar@{-}[u] &} 
\end{eqnarray*}
The signatures of these maximal twelve tone groups are as follows.
\begin{eqnarray*}
\SelectTips{cm}{11}\xymatrix{  &  {(4,3,4)}  & \\
 {(4,2,1,4)} \ar@{-}[ru] & {(2,2,3,2,2)} \ar@{-}[u] &  {(4,1,2,4)} \ar@{-}[lu]  \\
 {(2,2,2,1,2,2)} \ar@{-}[ru] \ar@{-}[u]  & {(4,1,1,1,4)}  \ar@{-}[ru] \ar@{-}[lu]& {(2,2,1,2,2,2)} \ar@{-}[lu] \ar@{-}[u] \\
 &  {(2,2,1,1,1,2,2)} \ar@{-}[ru] \ar@{-}[lu]  \ar@{-}[u] &} 
\end{eqnarray*}
The heptatonic scales for these groups are the sets
\begin{eqnarray*}
\SelectTips{cm}{11}\xymatrix{  &  {\TT_{7,\Delta_1}}  & \\
{\TT_{7,\Delta}}\ar@{-}[ru] & {\TT_{7,\Gamma_1}} \ar@{-}[u] &  {\TT_{7,\Delta_{-}}}\ar@{-}[lu]  \\
{\TT_{7,\Gamma}} \ar@{-}[ru] \ar@{-}[u]  &   {\TT_{7,\Delta_0}} = {\TT_{7,\Delta_1}} \ar@{-}[lu] \ar@{-}[ru] & {\TT_{7,\Gamma_{-}}} \ar@{-}[lu] \ar@{-}[u] \\
 &  {\TT_{7,\Gamma_0}} = {\TT_{7,\Gamma_1}} \ar@{-}[ru] \ar@{-}[lu]  \ar@{-}[u] &} 
\end{eqnarray*}
having cardinalities
\begin{eqnarray*}
\SelectTips{cm}{11}\xymatrix{  &  {108}  & \\
 {72} \ar@{-}[ru] & {48} \ar@{-}[u] &  {72} \ar@{-}[lu]  \\
 {32} \ar@{-}[ru] \ar@{-}[u]  & {108 = 36 \cdot 3} \ar@{-}[ru] \ar@{-}[lu] & {32} \ar@{-}[lu] \ar@{-}[u] \\
 & {48 = 16 \cdot 3} \ar@{-}[ru] \ar@{-}[lu]  \ar@{-}[u] &} 
\end{eqnarray*}
It is interesting that ${\TT_{7,\Delta_0}} = {\TT_{7,\Delta_1}}$ and ${\TT_{7,\Gamma_0}} = {\TT_{7,\Gamma_1}}$, even though  ${\TT_{7,\Delta_0}}$ and ${\TT_{7,\Gamma_0}}$ each consist of three equal-sized orbits (of size 36 and 16, respectively) while  ${\TT_{7,\Delta_1}}$ and ${\TT_{7,\Gamma_1}}$ each consist of a unique orbit (of size 108 and 48, respectively). 


While the analysis in this section provides some reasons for using the groups $\Gamma$ and $\Delta$, Theorem  \ref{mainconjecture} shows that $\Gamma$ is unique only up to conjugacy.  One ought to ask whether or not $\Gamma$ is the ``best'' choice among all of its conjugates $G$ for the scales in the maximal $G$-orbit. 

We believe that the choice of $\Gamma$ and $\Delta$, and thus the m\={e}\d{l}akarta raga system, can be justified.  All 72 of the m\={e}\d{l}akarta ragas contain $0$ and $7$, which is a natural restriction to impose as the interval $\{0,7\}$ is a {\bf fifth} (in fact, a {\it perfect fifth} in Indian classical tuning).  The $72 = 1 \cdot 6 \cdot 2 \cdot 1\cdot 6$ m\={e}\d{l}akarta ragas are obtained as follows:
$$\{0\} \cup \{\mbox{two of } 1,2,3,4\} \cup  \{\mbox{one of } 5,6\} \cup \{7\} \cup  \{\mbox{two of } 8,9,10,11\}.$$
These 72 scales comprise the maximal $\Delta$-orbit of $\TT_7$, where $\Delta$ is the subgroup of $S_{11}$ isomorphic to $S_4 \times S_2 \times S_1 \times S_4$ that acts separately on $\{1,2,3,4\}$, $\{5,6\}$, $\{7\}$, and $\{8,9,10,11\}$.  In other words, the 72 m\={e}\d{l}akarta ragas are precisely the scales with largest $(\Delta,t)$-musciality for any $t \in [-\infty,\infty]$.  The choice of $\Gamma$ and $\Delta$ among their conjugates are ``natural'' choices at the very least to the extent that the perfect fifth is ``natural.''
 
For any subgroup $G$ of $S_{11}$ conjugate to $\Gamma$, one is required to choose an element of $\ZZ_{12}$ besides $0$ that is globally fixed by the action of $G$.   To the extent that the perfect fifth is ``natural,'' the most natural choice is the element $7$, but one may rightfully choose  $5$ instead.   Either of these is a natural choice since, while the interval $\{0,7\}$ is a fifth, the interval $\{0,5\}$ is a {\bf fourth}, and both intervals coincide with perfect harmonic intervals (in some tunings).  Moreover, 5 and 7 are the only elements of the cyclic group $\ZZ_{12}$ other than $1$ and $-1 = 11$  that generate the whole group, a fact that forms the basis of the {\it circle of fifths} and {\it circle of fourths}.

Once the choice of a second fixed element of $\ZZ_{12}$ is made, one is required to partition the remaining ten elements of $\ZZ_{12}$ into disjoint two-element subsets $\{a_i,b_i\}$ for $i = 1,2,3,4,5$, where $G$ is then to act  separately on $\{a_i,b_i\}$ for all $i$.  A natural choice is for each $a_i$ and $b_i$ to be {\it consecutive}, so that $b_i = a_i \pm 1$ for all $i$.
This is because any other choice would require the action to be ``non-local,'' with the occurence of ``crossings,''  as, for example, if $G$ were to act separately on $\{1,4\}$ and $\{2,3\}$.  It is clear that every subgroup of $S_n$ is conjugate to a subgroup that acts without crossings.  There are exactly six conjugates of $\Gamma$ that act without crossings, namely, those that fix $1$, $3$, $5$, $7$, $9$, or $11$, respectively.  (By definition, all of them fix $0$.)  Among these six conjugates of $\Gamma$, the only one besides $\Gamma$ that it may also be natural to consider is the subgroup $\Gamma_-$ generated by $\{(1 \ 2),(3 \ 4),(6 \ 7),(8 \ 9),(10 \ 11)\}$, which fixes $5$ instead of $7$.  From this group $\Gamma_-$ we obtain the following  32 scales in the unique maximal $\Gamma_-$-orbit: 
$$\left\{\mbox{C}, {\mbox{D}\flat \atop \mbox{D}}, {\mbox{E}\flat \atop \mbox{E}}, \mbox{F}, {\mbox{G}\flat \atop \mbox{G}}, {\mbox{A}\flat \atop \mbox{A}}, {\mbox{B}\flat \atop \mbox{B}}\right\}.$$
Half of these 32 scales---those that contain G---were obtained previously using $\Gamma$, so among these 32 scales we obtain the 16 additional scales listed in Table 3, namely, those that contain G$\flat$ rather than G.
Thus, the union of the maximal $\Gamma$-orbit and the maximal $\Gamma_-$-orbit consists of $48 = 32+16 = 32+32-16$ scales. 


It is natural also to consider the compositum $\Gamma_1 = \Gamma \Gamma_-$ of $\Gamma$ and $\Gamma_-$, which is isomorphic to $S_2 \times S_2 \times S_3 \times S_2 \times S_2$ and acts without crossings, separately on $\{1,2\}$, $\{3,4\}$, $\{5,6,7\}$, $\{8,9\}$, and $\{10,11\}$.   The unique maximal $\Gamma_1$-orbit consists of the 48 scales listed in Tables 1 and 2 comprising the union of the maximal $\Gamma$-orbit and the maximal $\Gamma_-$-orbit.  The group $\Gamma_1$ has the virtue that its action on $\{1,2,3,\ldots,11\}$ is completely symmetrical, providing a natural theory of heptatonic scales that privileges both the fourth and the fifth equally.  Moreover, the group $\Gamma_1$ ranks 3rd--8th among the 3094 conjugacy classes of subgroups $G$ of $S_{11}$ for its value of $\diam(G,\TT_7)$ for $t = 1$ and 2nd--7th for its values for $t = 0$ and $t = -1$.   One has
$$\diam_1(\Gamma_1, \TT_7) \approx 3.0689,$$
$$\diam_0(\Gamma_1, \TT_7) \approx  4.3060,$$  $$\diam_{-1}(\Gamma_1, \TT_7) \approx  6.3377,$$
which closely rival the values for $\Gamma$ and $\Gamma_-$.
By contrast, the group $\Delta$ fares relatively poorly, ranking 529th--536th among the 3094 conjugacy classes for its value for $t = 1$, ranking 483rd--490th for its value for $t = 0$, and ranking 294th--301st for its value at $t = -1$.   These values are as follows:
$$\diam_1(\Delta, \TT_7) \approx 2.0182,$$
$$\diam_0(\Delta, \TT_7) \approx  2.5229,$$
$$\diam_{-1}(\Delta, \TT_7) \approx  3.8961.$$
By this measure, then, the 48 scales in the maximal $\Gamma_1$-orbit are a worthy alternative to the 72 (m\={e}\d{l}akarta) scales in the maximal $\Delta$-orbit.

\begin{table}
\caption{16 additional scales  in the maximal $\Gamma_1$-orbit of $\TT_7$}
\centering 
\begin{tabular}{l|lllllll} \hline
 & C & D & E & F & G$\flat$ & A & B \\ \hline
 & C & D & E & F & G$\flat$ & A & B$\flat$ \\ \hline
 & C & D & E & F & G$\flat$ & A$\flat$ & B \\ \hline
major Locrian & C & D & E & F & G$\flat$ & A$\flat$ & B$\flat$ \\ \hline
 & C & D & E$\flat$ & F & G$\flat$ & A & B \\ \hline
 & C & D & E$\flat$ & F & G$\flat$ & A & B$\flat$ \\ \hline
 & C & D & E$\flat$ & F & G$\flat$ & A$\flat$ & B \\ \hline
 half diminished & C & D & E$\flat$ & F & G$\flat$ & A$\flat$ & B$\flat$ \\ \hline
 & C & D$\flat$ & E & F & G$\flat$ & A & B \\ \hline
 & C & D$\flat$ & E & F & G$\flat$ & A & B$\flat$ \\ \hline
Persian & C & D$\flat$ & E & F & G$\flat$ & A$\flat$ & B \\ \hline
 & C & D$\flat$ & E & F & G$\flat$ & A$\flat$ & B$\flat$ \\ \hline
 & C & D$\flat$ & E$\flat$ & F & G$\flat$ & A & B \\ \hline
Locrian 6 & C & D$\flat$ & E$\flat$ & F & G$\flat$ & A & B$\flat$ \\ \hline 
 & C & D$\flat$ & E$\flat$ & F & G$\flat$ & A$\flat$ & B \\ \hline
Locrian mode & C & D$\flat$ & E$\flat$ & F & G$\flat$ & A$\flat$ & B$\flat$ \\ \hline
\end{tabular}
\end{table}

It must be noted that the maximal twelve tone groups that are conjugate to $\Gamma$, besides $\Gamma_{-}$, are those of signature $(1,2,2,2,2,2)$, $(2,1,2,2,2,2)$, $(2,2,2,2,1,2)$, and $(2,2,2,2,2,1)$, respectively.  These four groups are those that fix $1$, $3$, $9$, and $11$, respectively, instead of $7$ or $5$.  These six groups appear in three ``inverse'' pairs: $\Gamma$ and $\Gamma_{-}$ are inversions, as are those of signature $(1,2,2,2,2,2)$ and $(2,2,2,2,2,1)$, as are those of signature $(2,1,2,2,2,2)$ and $(2,2,2,2,1,2)$.   At this stage, one ought to seek further {\it mathematical} justification for the choice of fixing $7$ or $5$ over and above $1$, $3$, $9$, or $11$, based on more than just the naturality of fixing the perfect fifth ($0$ and) $7$.  Regarding this problem, one of the two reviewers wrote the following.
\begin{quote}
The proposed formal framework provides no actual explanation for the privileged role of $7$ and $5$ for the choice of the second fixed element.  In my view, this is a methodological weakness. Although the authors try to explain it as a ``most natural choice'' but all the reasoning is based on aspects that are external to the actual formal framework: tuning and generators of $\ZZ_{12}$.  If tuning was crucial, why does $4$ play no role in the model? If generators were important, why do $1$ or $1$1 lead to unmusical systems? To solve the issue, I believe, the model should be enhanced by an additional formal constraint resulting in disqualification of other choices for the second fixed element. Below, I theorize about one possible approach (evenness)$\ldots$.

As I challenged above, the choice of the other fixed element as 7 or 5 is based on ad hoc arguments. It would be much more elegant if an additional formal concept was introduced from which the two choices of 7 and 5 would formally follow. I think that some generalization of Clough’s concept of evenness could be a viable option$\ldots$. I think that some measure of ``average evenness'' for systems of scales could be introduced (and computed) and it would disqualify the other choices of the fixed element in the heptatonic scales. Additionally, evenness applies obviously even on the level of particular scales. This would provide a natural ordering of scales with diatonic scales (and so the anhemitonic pentatonic scales) being maximally even.
\end{quote}
No doubt this is a promising way to resolve the issue.  As mentioned in the introduction, it is possible that the methods of this paper can be combined synergistically with other ways of mathematically justifying the various musical scales.  We leave this to the interested reader to pursue further.

The reviewers also commented that one ought to try to generalize Theorem \ref{mainconjecture} to atonic scales, that is, to the action of $S_{12}$ on the set $\SS_7$ of all 7-note atonic scales in $\SS$.
Based on that suggestion, we used GAP and SAGE to verify the following theorem, in a manner similar to the proof of Theorem \ref{mainconjecture}.  Note that there are 10723 conjugacy classes of  subgroups of $S_{12}$.

\begin{theorem}[with James Allen]\label{mainconjecture5}
For all $t \in [-1,1]$, the subgroups $G$ of $S_{12}$ for which $\diam_t(G,\SS_7)$ is largest are the group generated by $\{(0\ 1), (2 \ 3),(4 \ 5),(7 \ 8),(9 \ 10)(6 \ 11), (9 \ 11) (6 \ 10)\}$,
along with its conjugate subgroups.   Moreover, for all $t \in [0,1]$, the subgroups  $G$ of $S_{12}$ for which $\diam_t(G,\SS_7)$ is second largest are the group $\Gamma$ generated by  $\{(1 \ 2),(3 \ 4),(5 \ 6),(8 \ 9),(10 \ 11)\}$, along with its conjugate subgroups.  
\end{theorem}

The theorem provides further justification that the group $\Gamma$ and its conjugates are ``natural'' choices for a theory of musical scales.  On the interval $[-1,0]$, as $t$ approaches $-1$, one other conjugacy class begins to surpass the group $\Gamma$ in $t$-power diameter, namely, the conjugates of the group generated by
$\{(0\ 1), (2 \ 3),(4 \ 5),(7 \ 8)(9 \ 10),(9 \ 10)(6 \ 11), (9 \ 11) (6 \ 10)\}$.    The values of $\diam_t(G,\TT_7)$ for $t = 1,0,-1$ for these three conjugacy classes of subgroups $G$ of $S_{12}$ are provided in Table 4.  For reasons explained earlier, the values on the interval $[0,1]$ are more critical than those on the interval $[-1,0]$.

\begin{table} 
\caption{The three conjugacy classes of subgroups of $S_{12}$ with largest $t$-power diameter}
\centering 
{\begin{tabular}{l|l|l|l|l}
Signature &  Maximal orbits  & $\diam_1(G,\TT_7)$ & $\diam_0(G,\TT_7)$ &  $\diam_{-1}(G,\TT_7)$  \\\hline \hline
$(4, 2, 2, 2, 2)$ & 1 of size 64  & 3.9501 & 5.5199 & 7.8384  \\ \hline
$(2, 2, 2, 2, 2, 1, 1)$ & 1 of size 32  & 3.7917 & 5.0929 & 6.9091 \\ \hline
$(4, 2, 2, 2, 2)$ & 1 of size 64 & 3.7183 & 5.0397 & 7.0303 \\ \hline
\end{tabular}}
\end{table}


\section{Hexatonic scales}

It is most common to obtain a hexatonic scale in one of the following three ways: (1) deleting a note from a given heptatonic scale (e.g., the major and minor hexatonic scales are obtained from the major and natural minor heptatonic scales by deleting the seventh note and the sixth note, respectively); (2) adding a note to a given pentatonic scale (e.g, the major and minor blues hexatonic scales are obtained from the major and minor pentatonic scales by adding an extra half step after the third note in each); and (3) combining three non-overlapping triads.  

A fourth way of obtaining a hexatonic scale from a heptatonic scale is as follows.  Define the {\bf (tonic) complement} $\overline{s}$ of a $k$-tonic scale $s$ to be the $(13-k)$-tonic scale $$\overline{s} = (\ZZ_{12}-s)\cup\{0\}.$$
Of course one has $\overline{\overline{s}} = s$ for all scales $s \in \TT$.  For any action $\cdot$ of $S_{11}$ on $\TT$, there is a {\bf (tonic) complementary action} $\overline{\cdot}$ of $S_{11}$  defined by
$$\sigma \overline{\cdot} s = \overline{\sigma \cdot \overline{s}}$$
for all $\sigma \in S_{11}$ and all $s \in \TT$.  Moreover, the induced action $\cdot$ on $\TT_k$ corresponds to the induced action $\overline{\cdot}$ on $\TT_{13-k}$.  As a consequence, our results on actions of the subgroups of $S_{11}$ on $\TT_7$ yield corresponding results on the actions of the subgroups of $S_{11}$ on $\TT_6$.   Thus, by Theorem \ref{mainconjecture} and complementarity, we have the following.

\begin{theorem}\label{mainconjecture2}
For all $t \in [-1,1]$, the subgroups $G$ of $S_{11}$ for which $\diam_t(G,\TT_6)$ is largest are the group $\Gamma$ generated by $\{(1 \ 2),(3 \ 4),(5 \ 6),(8 \ 9),(10 \ 11)\}$, along with its conjugate subgroups.  
\end{theorem}

Also by complementarity, each of the theories of heptatonic scales developed in Section 5 has a complementary theory of hexatonic scales: the sets of complements of the scales in the maximal orbits of a given action on $\TT_7$ are precisely the maximal orbits of the complementary action on $\TT_6$.  However, none of the 32 scales in the maximal orbit of $\TT_6$ under the action of our twelve tone group $\Gamma$ of signature $(2,2,2,1,2,2)$ contains the interval $\{0,7\}$ of a fifth.  Consequently, this particular group is perhaps not the most natural for yielding interesting hexatonic scales.  For this purpose we single out the subgroup $\Lambda$ of $S_{11}$ generated by the set $\{(2 \ 3),(4 \ 5),(6 \ 7),(8 \ 9),(10 \ 11)\}$, which is the maximal twelve tone group of signature $(1,2,2,2,2,2)$, and the subgroup $\Lambda'$ of $S_{11}$ generated by the set $\{(1 \ 2),(3 \ 4),(5 \ 6),(7 \ 8),(9 \ 10)\}$, which is the maximal twelve tone group of signature $(2,2,2,2,2,1)$.  Among the 32 scales 
$$\left\{\mbox{C}, {\mbox{D} \atop \mbox{E}\flat}, {\mbox{E} \atop \mbox{F}}, {\mbox{G} \atop \mbox{G}\flat}, {\mbox{A} \atop \mbox{A}\flat}, {\mbox{B} \atop \mbox{B}\flat}\right\}$$
in the maximal $\Lambda$-orbit of $\TT_6$ appear the whole tone scale $\{\mbox{C, D, E, G$\flat$, A$\flat$, B$\flat$}\}$ (the complement of Neopolitan major), the Prometheus scale $\{\mbox{C, D, E, G$\flat$, A, B$\flat$}\}$ (the complement of Neopolitan minor), and the augmented scale $\{\mbox{C, E$\flat$, E, G, A$\flat$, B}\}$.  Among the 32 scales 
$$\left\{\mbox{C}, {\mbox{D} \atop \mbox{D}\flat}, {\mbox{E} \atop \mbox{E}\flat}, {\mbox{F} \atop \mbox{G}\flat}, {\mbox{G} \atop \mbox{A}\flat}, {\mbox{A} \atop \mbox{B}\flat}\right\}$$ in the maximal $\Lambda'$-orbit of $\TT_6$ appear the whole tone scale, the major hexatonic scale $\{\mbox{C, D, E, F, G, A}\}$, the minor hexatonic scale $\{\mbox{C, D, E$\flat$, F, G, B$\flat$}\}$, and the tritone scale $\{\mbox{C, D$\flat$, E, G$\flat$, G, B$\flat$}\}$.   The whole tone scale is the only scale that lies in both sets of 32 scales.   Unfortunately,  the major and minor blues hexatonic scales do not appear in either set but rather have $\Lambda$-orbits and $\Lambda'$-orbits  of size $16$.

\section{Pentatonic scales}

The most prominent of the pentatonic scales are the five {\bf black-key pentatonic scales}  formed by the black keys of a piano: the major and minor and blues major and minor pentatonic scales and the Egyptian, or suspended, pentatonic scale.  Let $\Sigma$ denote the  subgroup of $S_{11}$ generated by  $\{(2 \ 3),(4 \ 5),(7 \ 8),(9 \ 10)\}$.  In other words, $\Sigma$ is the maximal twelve tone group of signature $(1,2,2,1,2,2,1)$.  The 16 scales in the unique maximal $\Sigma$-orbit of $\TT_5$ are the 16 scales 
$$\left\{\mbox{C}, {\mbox{D} \atop \mbox{D} \sharp}, {\mbox{E} \atop \mbox{F}}, {\mbox{G} \atop \mbox{G}\sharp}, {\mbox{A} \atop \mbox{A}\sharp}\right\}$$
listed in Table 5.  Among these 16 scales are the five black-key pentatonic scales. 

\begin{table}
\caption{16 pentatonic scales in the maximal $\Sigma$-orbit of $\TT_5$}
\centering 
\begin{tabular}{l|lllll} \hline
major  & C & D & E &   G & A \\ \hline
 & C & D & E &  G & A$\sharp$ \\ \hline
 & C & D & E &   G$\sharp$ & A \\ \hline
 & C & D & E &   G$\sharp$ & A$\sharp$ \\ \hline
Blues major, or Ritsusen, or yo & C & D & F &   G & A \\ \hline
Egyptian, or suspended  & C & D & F &   G & A$\sharp$ \\ \hline
 & C & D & F &   G$\sharp$ & A \\ \hline
& C & D & F &  G$\sharp$ & A$\sharp$ \\ \hline
 & C & D$\sharp$ & E &   G & A \\ \hline
 & C & D$\sharp$ & E &  G & A$\sharp$ \\ \hline
& C &  D$\sharp$ & E &  G$\sharp$ & A \\ \hline
 & C & D$\sharp$ & E &  G$\sharp$ & A$\sharp$ \\ \hline
 & C & D$\sharp$ & F &   G & A \\ \hline
minor & C & D$\sharp$ & F &  G & A$\sharp$ \\ \hline 
 & C & D$\sharp$ & F &  G$\sharp$ & A \\ \hline
Blues minor, or Man Gong & C & D$\sharp$ & F &   G$\sharp$ & A$\sharp$ \\ \hline
\end{tabular}
\end{table}

The following theorem was proved using GAP and SAGE in a manner similar to the proof of Theorems \ref{mainconjecture} and \ref{mainconjecture5}.

\begin{theorem}\label{mainconjecture3}
For all $t \in [-1,1]$, the subgroups $G$ of $S_{11}$ for which $\diam_t(G,\TT_5)$ is largest are the group $\Sigma_1$ generated by $\{(2 \ 3),(4 \ 5),(7 \ 8),(9 \ 10)(6 \ 11), (9 \ 11) (6 \ 10)\}$,
along with its conjugate subgroups.   Moreover, for all $t \in [0,1]$, the subgroups  $G$ of $S_{11}$ for which $\diam_t(G,\TT_5)$ is second largest are the group $\Sigma$ generated by  $\{(2 \ 3),(4 \ 5),(7 \ 8),(9 \ 10)\}$, along with its conjugate subgroups.  
\end{theorem}

On the interval $[-1,0]$, as $t$ approaches $-1$, several other subgroups begin to surpass the group $\Sigma$ in $t$-power diameter. 

The group $\Sigma_1$ of the theorem is isomorphic to $\ZZ_2^5$, and the group $\Sigma$ is a subgroup of $\Sigma_1$ isomorphic to $\ZZ_2^4$.   For the action of $\Sigma_1$ on $\TT_5$ there are:
\begin{enumerate}
\item $1$ orbit of size $32$,
\item $3$ orbits of size $16$,
\item $19$ orbits of size $8$,
\item $16$ orbits of size $4$,
\item $15$ orbits of size $2$,
\item $4$ orbits of size $1$,
\end{enumerate}
for a total of $58$ orbits.  For the action of $\Sigma$ on $\TT_5$ there are:
\begin{enumerate}
\item $1$ orbit of size $16$,
\item $12$ orbits of size $8$,
\item $20$ orbits of size $4$,
\item $40$ orbits of size $2$,
\item $18$ orbits of size $1$,
\end{enumerate}
for a total of $101$ orbits.   From this we deduce that
$$\diam_1(\Sigma_1,\TT_5) \approx  3.1731$$
$$\diam_0(\Sigma_1,\TT_5) \approx  4.2068$$
$$\diam_{-1}(\Sigma_1,\TT_5) \approx  5.6242$$
and
$$\diam_1(\Sigma,\TT_5) \approx  3.1391$$
$$\diam_0(\Sigma,\TT_5) \approx 3.9004$$
$$\diam_{-1}(\Sigma,\TT_5) \approx  4.8970.$$

One of the reviewers suggested the following theory of {\it atonic} pentatonic scales alternative to our theory of {\it tonic} pentatonic scales.
\begin{quote}
I really like the idea of complementarity. It is a very elegant way of dealing with related systems. However, from the musical perspective it seems quite counterintuitive that the hexatonic scales turn out to be the complements to the heptatonic scales. How a Neapolitan minor scale is complementary to the whole tone scale? Of course, it follows from the feature that one tone is fixed in all scales and only the others are movable.

However, if this feature is reconsidered, one might achieve an elegant explanation of the pentatonic scales while keeping the heptatonic scales in check. A scale would be any subset of $\ZZ_{12}$, not necessarily containing 0. Instead of $S_{11}$, one would consider actions of subgroups of $S_{12}$. The definition of local actions would require a cosmetic change: it would need to consider the cyclic nature of $\ZZ_{12}$.

As regards the heptatonic scales, I conjecture that one would obtain the maximal $t$-orbit diameters with the signature $(2,2,2,2,2,1,1)$. To fix $0$ one could consider the permutations of the signature starting with $1$. Then the evenness should lead to two resulting signatures $(1, 2, 2, 1, 2, 2, 2)$ and $(1, 2, 2, 2, 1, 2, 2)$. One could even call them the authentic and the plagal systems.

And for the pentatonic scales one should get a perfect remedy. The pentatonic scales would be the complementary scales to the heptatonic scales. Therefore, the maximal twelve tone groups are conjugates of $\Gamma$. It would be elegant to consider the conjugate group generated by $\{(11  \ 0), (2\  3), (4 \ 5), (7\ 8), (9 \ 10)\}$.  This makes $0$ movable, but leads to relevant musical scales. It is nice that this way the system includes not only the {\it yo} but also the  {\it in} pentatonic scales. (The Japanese music theorist Uehara proposed two basic pentatonic modes: {\it yo} (anhemitonic) and {\it in} (hemitonic).)
\end{quote}
Under the scheme of  pentatonic atonic scales proposed above, the largest orbit has $32$ elements, exactly half of which are tonic scales.


{}

\end{document}